\documentclass[a4paper,10pt]{article}
\usepackage[margin = 0.85 in]{geometry}
\usepackage{amsfonts,amsmath,amsthm,amssymb,indentfirst,placeins,float,graphicx,booktabs, verbatim,bbm,microtype,authblk,multicol}

\usepackage[hyperindex,breaklinks]{hyperref}
\hypersetup{linkcolor=blue, citecolor=red, colorlinks=true}
\usepackage[T1]{fontenc}
\usepackage{lmodern}
\pdfoutput=1
\allowdisplaybreaks
\theoremstyle{plain}
\newtheorem{myth}{Theorem}[section]
\newtheorem{mylm}[myth]{Lemma}
\newtheorem{mypr}[myth]{Proposition}

\theoremstyle{remark}
\newtheorem{mydef}[myth]{Definition}

\newcommand{\mc}{\mathcal}

\numberwithin{equation}{section}

\begin{document}
\title{\textbf{Backward stochastic differential equations with regime-switching and sublinear expectations}}
\author[1]{Engel John C. Dela Vega}
\author[1,2]{Robert J. Elliott\footnote{Corresponding Author; Email: relliott@ucalgary.ca}}
\affil[1]{UniSA Business, University of South Australia, Adelaide, SA 5000, Australia}
\affil[2]{Haskayne School of Business, University of Calgary, Calgary, Alberta, Canada}
\date{}
	\maketitle
	\vspace{-1cm}
	\begin{abstract}
		This paper introduces a backward stochastic differential equation driven by both Brownian motion and a Markov chain (BSDEBM). Regime-switching is also incorporated through its driver. The existence and uniqueness of the solution of the BSDEBM are proved. A comparison theorem is also derived. Filtration consistent sublinear expectations are defined and characterized as solutions to the BSDEBM. The bid and ask prices are then represented using sublinear expectations.
		
		\vspace{12pt}
		
	\end{abstract}

	\section{Introduction}
	
	The theory of backward stochastic differential equations (BSDEs) has been broadly studied because of its connections to stochastic control \cite{hamadene:control, zhang:bsde,yongmin:controlbook, LiWangWu:bsdeinoptimalcontrol, ref5:opt} and mathematical finance \cite{ref4:opt, elkarouiquenez:bsdeinfinance, eyraudloisel:bsdeinoptionhedging}. Linear BSDEs were first introduced as adjoint equations of a stochastic control problem in \cite{bismut:linearBSDE}. The more general nonlinear BSDEs were then introduced in \cite{pardoux:bsde}, where the existence and uniqueness of solutions were also discussed and proved. The BSDEs studied in \cite{pardoux:bsde} are driven by Brownian motion alone. Jump diffusions were then also considered, such as BSDEs driven by L\'{e}vy processes \cite{bahlali:bsdewithlevy,nualartschoutens:bsdewithlevy}, and BSDEs driven by a Brownian motion and a Poisson random measure \cite{pardoux:bsdewithpoisson,tangli:bsdewithpoisson}. BSDEs driven by Markov chains have also been explored \cite{cohen:bsde2, cohen:bsde}.
	
	A classic result of the theory of BSDEs is the comparison theorem which was first introduced in \cite{peng:comparison}, where the BSDEs are driven by Brownian motion. As an extension, comparison theorems for BSDEs with jumps have also been discussed \cite{pardoux:bsdewithpoisson, elotmani:comparison, lin:comparison}. Another result is the connection of BSDEs with the theory of nonlinear expectations \cite{cohen:bsde, cohen:generalcomparison}. Nonlinear expectations are used mainly in problems dealing with uncertainty \cite{peng:sublinear}. They are related to the theory of $G$-expectations pioneered by Peng \cite{peng:g, ref1:sublinear, peng:gmulti}. It was shown in \cite{ref7:sublinear} that nonlinear expectations can be represented as a solution to a BSDE if the driver satisfies some conditions.
	
	One type of nonlinear expectation is the sublinear expectation, which satisfies subadditity and positive homogeneity. Sublinear expectations can be represented as a supremum of a family of expectations \cite{peng:sublinear,ref2:sublinear}. The theory of sublinear expectations has been known to provide several tools in finance and hence has been explored in more detail in both theory and in applications \cite{ref6:sublinear, ref3:sublinear, ref4:sublinear, ref5:sublinear}.
	
	In relation to BSDEs, regime-switching drivers have also been introduced in the literature. A two-time scale associated to weak and strong interactions for the Markov chain is discussed in \cite{taowuzhang:bsderegimeswitching}. A backward doubly stochastic differential equation with regime-switching is investigated in \cite{mawu:doublyswitchingdriver}. Singularly perturbed Markov chains incorporated in the driver are discussed in \cite{TangWu:bsdewithswitch}. Regime-switching BSDEs driven by Brownian motion and a Poisson random measure is explored in \cite{ShiWu:bsdewithswitch}. Similar results are discussed under the framework of reflected and doubly reflected BSDEs in \cite{crepeymatoussi:reflecteddoubly}. Regime-switching is also introduced in the driver and dynamics of a multi-valued BSDE in \cite{dengren:multibsderegime}.
	
	
	In this paper, we introduce a backward stochastic differential equation with a regime-switching driver driven by both a Brownian motion and a Markov chain (BSDEBM). Existence and uniqueness of the solution of the BSDEBM are established by proving the results for a sequence of three BSDEBMs of increasing complexity in the driver. A set of conditions is presented to allow a comparison theorem for the BSDEBM to hold. Filtration consistent sublinear evaluations and expectations are also defined and represented as solutions of the BSDEBM subject to some assumptions. The bid and ask prices are then consequently represented as solutions of the BSDEBM.
	
	Markov jump martingales can be decomposed into sums of marked Poisson processes as discussed in \cite{yinzhu:switchingdiffusions}. However, the canonical semimartingale representation of a Markov chain in \cite{elliott:hmm} is used in the results of this paper. The use of the semimartingale representation of the chain is in itself unique since it makes the calculations more tractable and insightful. To highlight the differences between the use of Markov chains as noise instead of Poisson processes, we mention the main differences in the results between this paper and \cite{ShiWu:bsdewithswitch}: (1) the technique used in proving the existence and uniqueness of solutions and (2) the condition imposed on the driver for the comparison theorem to hold. In this paper, three stages of a Picard-type iteration are used to prove existence and uniqueness of the solution to the BSDEBM, and the condition on the driver is specific to the Markov chain's rate matrix and associated seminorm. In \cite{ShiWu:bsdewithswitch}, a contraction mapping theorem is used to prove existence and uniqueness of the solution to their BSDE, and the condition on the driver involves two distinct Markov chains. A similar comparison can be made between Markov chains and L\'{e}vy processes.
	
	Furthermore, this paper extends the results of BSDEs driven by Brownian motion such as in \cite{pardoux:bsde} and BSDEs driven by a Markov chain as in \cite{cohen:bsde}. Markovian regime-switching is also introduced in the driver since a change in regime should have a corresponding change in the driver. We also impose that the transition rates of the Markov chain are bounded by a specific constant. We also use a special type of nonlinear expectations, which are the sublinear expectations, to represent them as solutions to the BSDEBM and to characterize bid and ask prices. 
	
	The paper is organized as follows: the dynamics of the BSDEBM are introduced in Section \ref{preliminaries}. The existence and uniqueness of the solutions of the BSDEBM are proved in Section \ref{existence and uniqueness of chapter 5}. A comparison theorem in Section \ref{comparison theorem chapter 5} is then derived. Section \ref{sublinear evaluations and expectations} defines the sublinear expectations and provides a representation of the bid and ask prices.

	\section{Preliminaries}\label{preliminaries}	
	
	Consider a complete probability space $(\Omega,\mc{F},P)$, where $P$ is a real-world probability measure. Let $\mathbf{W}:=\{W_t\}_{t\in [0,T]}$ be a standard Brownian motion on $(\Omega,\mc{F},P)$. Assume that states in an economy are modelled by a continuous-time Markov chain $\mathbf{X}:=\{X_t\}_{t\in[0,T]}$ on $(\Omega,\mc{F},P)$ with finite state space $\mc{S}:=\{e_1,e_2,\ldots,e_N \}\subset \mathbb{R}^N$, where $e_i\in\mathbb{R}^N$ and the $j$th component of $e_i$ is the Kronecker delta $\delta_{ij}$ for each $i,j=1,\ldots,N$. Initially assume that the path of $\mathbf{X}$ is known until time $T$. In addition, we suppose $\mathbf{X}$ and $\mathbf{W}$ are independent processes.
	
	Consider a filtration $\mathbb{F}:=\{\mc{F}_t\}_{t\in[0,T]}$ such that $\mc{F}_t\subseteq\mc{F}$ for all $t\in[0,T]$ and $\mc{F}_T=\mc{F}$. Write $\mathbb{F}^{\mathbf{X}}:=\{\mc{F}_t^{\mathbf{X}}\}_{t\in[0,T]}$ and $\mathbb{F}^{\mathbf{W}}:=\{\mc{F}_t^{\mathbf{W}}\}_{t\in[0,T]}$ for the $P$-augmented natural filtrations generated by $\mathbf{X}$ and $\mathbf{W}$, respectively. For each $t\in[0,T]$, we define $\mc{G}_t:=\mc{F}_t^{\mathbf{X}}\lor \mc{F}_t^{\mathbf{W}}$, the minimal augmented $\sigma$-field generated by the two $\sigma$-fields $\mc{F}_t^{\mathbf{X}}$ and $\mc{F}_t^{\mathbf{W}}$. Write $\mathbb{G}:=\{\mc{G}_t\}_{t\in[0,T]}$.
	
	Let $A:=[a_{ij}]_{N\times N}$ be a time-independent rate matrix of the chain $\mathbf{X}$ under $P$, where for $i\neq j$, $a_{ji}$ is the transition rate of the chain jumping from state $e_i$ to $e_j$. In \cite{elliott:hmm}, the following semimartingale dynamics for $\mathbf{X}$ under $P$ are obtained:
	\begin{equation}\label{chain section 5}
		X_t=X_0+\int_{0}^{t} AX_sds+M_t,
	\end{equation}
	where $\mathbf{M}:=\{M_t\}_{t\in [0,T]}$ is an $\mathbb{R}^N$-valued, square-integrable, $(\mathbb{F}^{\mathbf{X}},P)$-martingale. We further assume that there exists $c\in(0,1)$ such that $a_{ij}\in[c,c^{-1}]$ for all $i,j$ with $i\neq j$.
	
	Consider a backward stochastic differential equation driven by both Brownian motion and a Markov chain (BSDEBM) of the form
	\begin{align}\label{bsde}
		Y_t-\int_t^TF(u,Y_u,Z_1(u),Z_2(u),X_u)du+\int_t^TZ_1(u)dW_u+\int_t^TZ^{\top}_2(u)dM_u=Q,
	\end{align}
	where, for each $t\in[0,T]$, $F(\cdot,t,\cdot,\cdot,\cdot,\cdot):\Omega\times\mathbb{R}\times\mathbb{R}\times\mathbb{R}^N\times \mc{S}\to\mathbb{R}$, $Q:L^2(\mathbb{R};\mc{G}_T)\to\mathbb{R}$, and $Z_2(t)=(Z^1_2(t),\ldots,Z^N_2(t))^{\top}\in\mathbb{R}^N$ with $Z^i_2(t):\Omega\times[0,T]\to\mathbb{R}$. Furthermore, we suppose the driver $F$ in \eqref{bsde} satisfies the following assumptions:
	
	\begin{enumerate}
		\item[i.] For all $(y,z_1,z_2,x)\in\mathbb{R}\times\mathbb{R}\times\mathbb{R}^N\times\mc{S}$, $F(\cdot,y,z_1,z_2,x)\in M^2(0,T;\mathbb{R},\mc{G}_t)$.
		\item[ii.] There exists a constant $\mu>0$ such that for all $(\omega,t,x)\in\Omega\times[0,T]\times\mc{S}$, and triples $(y,z_1,z_2),(y',z'_1,z'_2)\in\mathbb{R}\times\mathbb{R}\times\mathbb{R}^N$,
		\begin{align*}
			|F(t,y,z_1,z_2,x)-F(t,y',z'_1,z'_2,x)|^2&\leq \mu(|y-y'|^2+|z_1-z'_1|^2\\
			&\qquad+\lVert z_2-z'_2\rVert^2_{X_{t-}}),
		\end{align*}
		where $\lVert C\rVert^2_{X_t}:=C^{\top}\psi_t C$ and $\psi_t$ is a nonnegative definite matrix defined by
		\begin{align}\label{defn of psi}
			\psi_t:=\mbox{diag}(AX_{t-})-A\mbox{diag}(X_{t-})-\mbox{diag}(X_{t-})A^{\top}.
		\end{align}
	\end{enumerate}
	
	Denote the spaces $L^p(\mathbb{R}^K,\mc{G}_t)$ be the set of $\mathbb{R}^K$-valued $\mc{G}_t$-measurable, $p$-integrable random variables, $M^2(0,T;\mathbb{R}^{K\times N},\mc{G}_t)$ be the set of $\mathbb{R}^{K\times N}$-valued $\mc{G}_t$-adapted square-integrable processes over $\Omega\times[0,T]$, and $P^2(0,T;\mathbb{R}^{K},\mc{G}_t)$ be the set of $\mathbb{R}^K$-valued $\mc{G}_t$-predictable processes $\{\varphi_t\}_{t\in[0,T]}$ such that $E(\int_0^t\lVert\varphi_s\rVert_{X_{s-}}^2ds)<\infty$. A solution $(Y,Z_1,Z_2)$ to the BSDEBM in \eqref{bsde} is a triple 
	\begin{align*}
		(Y,Z_1,Z_2)\in M^2(0,T;\mathbb{R}^{K\times N},\mc{G}_t)\times M^2(0,T;\mathbb{R}^{K\times N},\mc{G}_t)\times P^2(0,T;\mathbb{R}^{K},\mc{G}_t).
	\end{align*}

	\section{Existence and Uniqueness}\label{existence and uniqueness of chapter 5}
	
	In this section, the existence and the uniqueness of the solution to the BSDEBM in \eqref{bsde} are established. A double martingale representation, which shows that a $(\mathbb{G},P)$-local martingale can be represented by unique $\mathbb{G}$-predictable processes, is presented. This representation is used to prove an existence and uniqueness theorem. 
	
	\begin{mylm}\label{number of jumps j to i}
		For each $i\neq j$, let $N^{ij}_t$ be the number of jumps from state $e_i$ to $e_j$ until time $t\in[0,T]$. Then
		\begin{align*}
			N^{ij}_t=\int_0^ta_{ji}\langle X_s,e_i\rangle ds+M^{ij}_t,
		\end{align*}
		where $M^{ij}:=\{M^{ij}_t\}_{t\in[0,T]}$ is an $(\mathbb{F}^{\mathbf{X}},P)$-martingale given by
		\begin{align*}
			M^{ij}_t=\int_0^t\langle X_{s-},e_i\rangle\langle dM_s,e_j\rangle.
		\end{align*}
	\end{mylm}
	\begin{proof}
		By definition,
		\begin{align*}
			N^{ij}_t:=\sum_{0<s\leq t}\langle X_{s-},e_i \rangle\langle X_{s},e_j \rangle=\sum_{0<s\leq t}\langle X_{s-},e_i \rangle\langle \Delta X_{s},e_j \rangle.
		\end{align*}
		From \eqref{chain section 5},
		\begin{align*}
			N^{ij}_t&=\sum_{0<s\leq t}\langle X_{s-},e_i \rangle\langle \Delta X_{s},e_j \rangle\\
			&=\int_0^t\langle X_{s-},e_i \rangle\langle dX_{s},e_j \rangle\\
			&=\int_0^t\langle X_{s-},e_i \rangle\langle AX_{s},e_j \rangle ds+\int_0^t\langle X_{s-},e_i \rangle\langle dM_s,e_j \rangle\\
			&=\int_0^ta_{ji}\langle X_{s-},e_i \rangle ds+M^{ij}_t.
		\end{align*}
		Since $\mathbf{M}$ is an $(\mathbb{F}^{\mathbf{X}},P)$-martingale, then $M^{ij}$ is also an $(\mathbb{F}^{\mathbf{X}},P)$-martingale.
	\end{proof}
	
	\begin{myth}\label{double martingale representation}
		For any $(\mathbb{G},P)$-local martingale $R:=\{R_t\}_{t\in[0,T]}$, there exist unique $\mathbb{G}$-predictable processes $\phi_1:=\{\phi_1(t)\}_{t\in[0,T]}$ and $\phi_2:=\{\phi_2(t)\}_{t\in[0,T]}$, where $\phi_{2}(t)=[\phi_2^{ij}(t)]_{N\times N}$ such that
		\begin{align*}
			R_t=R_0+\int_0^t\phi_1(s)dW_s+\sum_{i\neq j}\sum_{j=1}^N\int_0^t\phi^{ij}_2(s)dM^{ij}_s.
		\end{align*}
	\end{myth}
	\begin{proof}
		The proof follows from Theorem 4.15 of \cite{elliott:doublemartingales} or Theorem 2.1 of \cite{boelkohlmann:doublemartingales}.
	\end{proof}
	
	Before proving the existence and uniqueness of the solution to the main BSDEBM, we first consider simplified versions of $\eqref{bsde}$. The first BSDEBM below has a driver that is dependent only on $t$ and $X_t$. 
	
	\begin{mypr}\label{simplified BSDE}
		Consider a simplified version of \eqref{bsde} of the form
		\begin{align}\label{bsde2}
			Y_t-\int_t^TF(u,X_u)du+\int_t^TZ_1(u)dW_u+\int_t^TZ^{\top}_2(u)dM_u=Q.
		\end{align}
		Then \eqref{bsde2} has a unique solution $(Y,Z_1,Z_2)$.
	\end{mypr}
	\begin{proof}
		Define 
		\begin{align*}
			Y_t=E\left[Q+\int_t^TF(u,X_u)du\bigg|\mc{G}_t\right].
		\end{align*}
		Then,
		\begin{align*}
			Y_t+\int_0^tF(u,X_u)du=E\left[Q+\int_0^TF(u,X_u)du\bigg|\mc{G}_t\right],
		\end{align*}
		which is a square-integrable martingale by the assumptions made on $Q$ and $F$. It follows from Theorem \ref{double martingale representation} that there exist $\mathbb{G}$-predictable processes $Z_1:=\{Z_1(t)\}_{t\in[0,T]}$ and $\phi_2:=\{\phi_2(t)\}_{t\in[0,T]}$ such that
		\begin{align}\label{mrt}
			Y_t+\int_0^tF(u,X_u)du&=Y_0+\int_0^tZ_1(u)dW_u+\sum_{i\neq j}\sum_{j=1}^N\int_0^t\phi^{ij}_2(s)dM^{ij}_s.
		\end{align}
		Following the proof of Lemma 3.1 in \cite{cohen:bsde2}, define the predictable matrix process $\Phi$ by
		\begin{align*}
			[\Phi_t]_{ij}=\phi^{ij}_{2}(t)\cdot 1_{\{i\neq j\}}.
		\end{align*}
		Using Lemma \ref{number of jumps j to i}, the last term of the right-hand side of \eqref{mrt} can be written as
		\begin{align*}
			\sum_{i\neq j}\sum_{j=1}^N\int_0^t\phi^{ij}_2(s)dM^{ij}_s&=\int_0^t\sum_{i\neq j}\sum_{j=1}^N\phi^{ij}_2(s)\langle X_{s-},e_i \rangle\langle dM_s,e_j \rangle\\
			&=\int_0^t\sum_{i=1}^N\sum_{j=1}^N\phi^{ij}_2(s)\cdot 1_{\{i\neq j\}}\langle X_{s-},e_i \rangle\langle dM_s,e_j \rangle\\
			&=\int_0^tX^{\top}_{s-}\Phi_sdM_s=\int_0^tZ^{\top}_2(s)dM_s,
		\end{align*}
		where $Z_2(t)=X^{\top}_{t-}\Phi_t$. 
		
		By construction, $Y_T=Q$ and hence,
		\begin{align*}
			Y_t-\int_t^TF(u,X_u)du+\int_t^TZ_1(u)dW_u+\int_t^TZ^{\top}_2(u)dM_u=Q.
		\end{align*}
		
		Suppose the triples $(Y^1_t,Z^1_1(t),Z^1_2(t))$ and $(Y^2_t,Z^2_1(t),Z^2_2(t))$ both solve \eqref{bsde2}. Then for all $t\in[0,T]$,
		\begin{align*}
			Y^1_t-Y^2_t+\int_t^T[Z_1^1(u)-Z_1^2(u)]dW_u+\int_t^T[Z_2^1(u)-Z_2^2(u)]^{\top}dM_u=0.
		\end{align*}
		Taking a $\mc{G}_t$ conditional expectation implies that $Y^1_t=Y_t^2$ $P$-a.s for all $t$. Uniqueness from Theorem \ref{double martingale representation} implies that $Z_1^1(t)=Z_1^2(t)$ and uniqueness from the martingale representation for the chain in \cite{cohen:bsde2} implies that $Z_2^1(t)=Z_2^2(t)$ $d\langle M,M \rangle_t\times dP$-a.s, with $\langle M,M \rangle_t$ defined as the predictable quadratic variation of the chain.
	\end{proof}
	
	The second BSDEBM below has a driver similar to \eqref{bsde}, but does not depend on $Y_t$. The remaining results in this section have proofs that are divided into two parts: uniqueness and existence. The uniqueness part utilizes the Lipschitz continuity assumption and the existence part makes use of a Picard-type iteration, similar to the methods used in \cite{cohen:bsde2} and \cite{pardoux:bsde}.
	
	\begin{mypr}\label{BSDE with Z_1 and Z_2}
		Consider the BSDEBM of the form
		\begin{align}\label{bsde3}
			Y_t-\int_t^TF(u,Z_1(u),Z_2(u),X_u)du+\int_t^TZ_1(u)dW_u+\int_t^TZ^{\top}_2(u)dM_u=Q.
		\end{align}
		For each $(t,x)\in[0,T]\times\mc{S}$, suppose that there exists a $\mu>0$ such that
		\begin{align*}
			\left|F(t,Z^1_1(t),Z^1_2(t),x)-F(t,Z^2_1(t),Z^2_2(t),x)\right|^2&\leq\mu\left(|Z^1_1(t)-Z^2_1(t)|^2\right.\\
			&\qquad\left.+\lVert Z^1_2(t)-Z^2_2(t)\rVert_{X_{t-}}^2\right).
		\end{align*}
		Then \eqref{bsde3} has a unique solution $(Y,Z_1,Z_2)$.
	\end{mypr}
	\begin{proof}
		(Uniqueness.) Suppose $(Y^1,Z^1_1,Z^1_2)$ and $(Y^2,Z^2_1,Z^2_2)$ are both solutions to \eqref{bsde3}. Then,
		\begin{align*}
			Y^1_t-Y^2_t&=\int_t^T\left[F(u,Z^1_1(u),Z^1_2(u),X_u)-F(u,Z^2_1(u),Z^2_2(u),X_u)\right]du\\
			&\quad-\int_t^T\left[Z^1_1(u)-Z^2_1(u)\right]dW_u-\int_t^T\left[Z^1_2(u)-Z^2_2(u)\right]^{\top}dM_u.
		\end{align*}
		Using Ito's formula,
		\begin{align*}
			\left|Y^1_t-Y^2_t\right|^2&=-2\int_t^T\left[Y^1_u-Y^2_u\right]\left[F(u,Z^1_1(u),Z^1_2(u),X_u)\right.\\
			&\qquad\left.-F(u,Z^2_1(u),Z^2_2(u),X_u)\right]du\\
			&\quad-2\int_t^T\left[Y^1_u-Y^2_u\right]\left[Z^1_1(u)-Z^2_1(u)\right]dW_u\\
			&\quad-\int_t^T|Z^1_1(u)-Z^2_1(u)|^2du\\
			&\quad-2\int_t^T\left[Y^1_u-Y^2_u\right]\left[Z^1_2(u)-Z^2_2(u)\right]^{\top}dM_u\\
			&\quad-\sum_{t<u\leq T}|\Delta Y^1_u-\Delta Y^2_u|^2.
		\end{align*}
		Taking expectations,
		\begin{align*}
			E\left|Y^1_t-Y^2_t\right|^2&=-2\int_t^TE\left\{\left[Y^1_u-Y^2_u\right]\left[F(u,Z^1_1(u),Z^1_2(u),X_u)\right.\right.\\
			&\quad\qquad\left.\left.-F(u,Z^2_1(u),Z^2_2(u),X_u)\right]\right\}du\\
			&\quad-\int_t^TE|Z^1_1(u)-Z^2_1(u)|^2du-E\left(\sum_{t<u\leq T}|\Delta Y^1_u-\Delta Y^2_u|^2\right)\\
			&=-2\int_t^TE\left\{\left[Y^1_u-Y^2_u\right]\left[F(u,Z^1_1(u),Z^1_2(u),X_u)\right.\right.\\
			&\quad\qquad\left.\left.-F(u,Z^2_1(u),Z^2_2(u),X_u)\right]\right\}du\\
			&\quad-\int_t^TE|Z^1_1(u)-Z^2_1(u)|^2du\\
			&\quad-E\left(\sum_{t<u\leq T}\left|(Z^1_2(u)-Z^2_2(u))^{\top}\Delta M_u\right|^2\right)\\
			&=-2\int_t^TE\left\{\left[Y^1_u-Y^2_u\right]\left[F(u,Z^1_1(u),Z^1_2(u),X_u)\right.\right.\\
			&\quad\qquad\left.\left.-F(u,Z^2_1(u),Z^2_2(u),X_u)\right]\right\}du\\
			&\quad-\int_t^TE|Z^1_1(u)-Z^2_1(u)|^2du\\
			&\quad-\int_t^TE\left\lVert Z^1_2(u)-Z^2_2(u)\right\rVert_{X_{u-}}^2du.
		\end{align*}
		For any $k\in\mathbb{R}$, $\pm 2ab\leq k^{-2}a^2+k^2b^2$. By the Lipschitz assumption, there exists a $\mu>0$ such that
		\begin{align*} 
			E\left|Y^1_t-Y^2_t\right|^2&\leq\int_t^T\bigg[k^{-2}E\left|Y^1_u-Y^2_u\right|^2\\
			&\quad\qquad+k^2E\left|F(u,Z^1_1(u),Z^1_2(u),X_u)-F(u,Z^2_1(u),Z^2_2(u),X_u)\right|^2\\
			&\quad\qquad-E|Z^1_1(u)-Z^2_1(u)|^2-E\left\lVert Z^1_2(u)-Z^2_2(u)\right\rVert_{X_{u-}}^2\bigg]du\\
			&\leq\int_t^T\bigg[k^{-2}E\left|Y^1_u-Y^2_u\right|^2+(\mu k^2-1)\bigg(E|Z^1_1(u)-Z^2_1(u)|^2\\
			&\quad\qquad+E\left\lVert Z^1_2(u)-Z^2_2(u)\right\rVert_{X_{u-}}^2\bigg)\bigg]du.
		\end{align*}
		By the method of integrating factors,
		\begin{align*} 
			e^{t/{k^2}}E\left|Y^1_t-Y^2_t\right|^2&\leq(\mu k^2-1)\int_t^Te^{u/{k^2}}\bigg[E|Z^1_1(u)-Z^2_1(u)|^2\\
			&\qquad\qquad\qquad\qquad+E\left\lVert Z^1_2(u)-Z^2_2(u)\right\rVert_{X_{u-}}^2\bigg]du.\\
		\end{align*}
		Set $k^2\leq 1/\mu$. It follows that $E\left|Y^1_t-Y^2_t\right|^2=0$, which implies that $Y^1_t=Y^2_t$ $P$-a.s for all $t\in[0,T]$. Moreover,
		\begin{align*}
			\int_t^T\bigg[E|Z^1_1(u)-Z^2_1(u)|^2+E\left\lVert Z^1_2(u)-Z^2_2(u)\right\rVert_{X_{u-}}^2\bigg]du=0,
		\end{align*}
		which implies that $Z^1_1(t)=Z^2_1(t)$ $P$-a.s and $Z^1_2(t)=Z^2_2(t)$ $d\langle M,M \rangle_t\times dP$-a.s for all $t\in[0,T]$. Hence, $(Y^1,Z^1_1,Z^1_2)$ and $(Y^2,Z^2_1,Z^2_2)$ are equivalent and \eqref{bsde3} has a unique solution.
		
		(Existence.) Define an approximating sequence using a Picard-type iteration. Suppose that the triple $(Y^{n+1}_t,Z^{n+1}_1(t),Z^{n+1}_2(t))$ is the solution to the following BSDEBM:
		\begin{align}\label{bsde 3.1}
			Q&=Y^{n+1}_t-\int_t^TF(u,Z^{n}_1(u),Z^{n}_2(u),X_u)du+\int_t^TZ^{n+1}_1(u)dW_u\nonumber\\
			&\quad+\int_t^T(Z^{n+1}_2(u))^{\top}dM_u.
		\end{align}
		
		From Proposition \ref{simplified BSDE}, \eqref{bsde 3.1} has to have a unique solution. For simplicity of notation, we will usually suppress the time variable $u$ below. Using similar methods to those above,
		\begin{align*}
			E\left|Y^{n+1}_t-Y^{n}_t\right|^2&=-2\int_t^TE\bigg\{\left[Y^{n+1}_u-Y^{n}_u\right]\left[F(u,Z^{n}_1,Z^{n}_2,X_u)\right.\\
			&\qquad\quad\left.-F(u,Z^{n-1}_1,Z^{n-1}_2,X_u)\right]\bigg\}du\\
			&\quad-\int_t^TE|Z^{n+1}_1(u)-Z^{n}_1(u)|^2du\\
			&\quad-\int_t^TE\left\lVert Z^{n+1}_2(u)-Z^{n}_2(u)\right\rVert_{X_{u-}}^2du.
		\end{align*}
		Then for $k\in\mathbb{R}$,
		\begin{align*} 
			E\left|Y^{n+1}_t-Y^{n}_t\right|^2&\leq\int_t^T\bigg[k^{-2}E\left|Y^{n+1}_u-Y^{n}_u\right|^2\\
			&\quad+k^2E\left|F(u,Z^{n}_1,Z^{n}_2,X_u)-F(u,Z^{n-1}_1,Z^{n-1}_2,X_u)\right|^2\\
			&\quad-E|Z^{n+1}_1(u)-Z^{n}_1(u)|^2-E\left\lVert Z^{n+1}_2(u)-Z^{n}_2(u)\right\rVert_{X_{u-}}^2\bigg]du\\
			&\leq\int_t^T\bigg[k^{-2}E\left|Y^{n+1}_u-Y^{n}_u\right|^2\\
			&\quad+\mu k^2\left(E|Z^{n}_1(u)-Z^{n-1}_1(u)|^2\right.\\
			&\quad\left.+E\left\lVert Z^{n}_2(u)-Z^{n-1}_2(u)\right\rVert_{X_{u-}}^2\right)\\
			&\quad-E|Z^{n+1}_1(u)-Z^{n}_1(u)|^2-E\left\lVert Z^{n+1}_2(u)-Z^{n}_2(u)\right\rVert_{X_{u-}}^2\bigg]du.
		\end{align*}
		Using the method of integrating factors,
		\begin{align}\label{cauchy1} 
			e^{t/{k^2}}E\left|Y^{n+1}_t-Y^{n}_t\right|^2&\leq\int_t^Te^{u/{k^2}}\bigg[\mu k^2\left(E|Z^{n}_1(u)-Z^{n-1}_1(u)|^2\right.\nonumber\\
			&\qquad\left.+E\left\lVert Z^{n}_2(u)-Z^{n-1}_2(u)\right\rVert_{X_{u-}}^2\right)\nonumber\\
			&\qquad-E|Z^{n+1}_1(u)-Z^{n}_1(u)|^2\nonumber\\
			&\qquad-E\left\lVert Z^{n+1}_2(u)-Z^{n}_2(u)\right\rVert_{X_{u-}}^2\bigg]du.
		\end{align}
		
		Set $k^2\leq (2\mu)^{-1}$. Since $e^{t/{k^2}}E\left|Y^{n+1}_t-Y^{n}_t\right|^2\geq 0$,
		\begin{align*}
			&\int_t^Te^{u/{k^2}}\bigg[E|Z^{n+1}_1(u)-Z^{n}_1(u)|^2+E\left\lVert Z^{n+1}_2(u)-Z^{n}_2(u)\right\rVert_{X_{u-}}^2\bigg]du\\
			&\leq\frac{1}{2}\int_t^Te^{2\mu u}\bigg[E|Z^{n}_1(u)-Z^{n-1}_1(u)|^2+E\left\lVert Z^{n}_2(u)-Z^{n-1}_2(u)\right\rVert_{X_{u-}}^2\bigg]du\\
			&\leq\frac{1}{2^n}\int_t^Te^{2\mu u}\bigg[E|Z^{1}_1(u)-Z^{0}_1(u)|^2+E\left\lVert Z^{1}_2(u)-Z^{0}_2(u)\right\rVert_{X_{u-}}^2\bigg]du.
		\end{align*}
		This implies that $\{Z^n_1(t)\}$ and $\{Z^n_2(t)\}$ are Cauchy sequences under appropriate norms. Using \eqref{cauchy1}, $Y^n_t$ is also a Cauchy sequence. By completeness, the limits exist and thus satisfy \eqref{bsde3}. 
	\end{proof}
	
	The following theorem proves the existence and uniqueness of the solution to the BSDEBM in \eqref{bsde}. The proof is quite similar to the proposition above.
	
	\begin{myth}[Existence and Uniqueness]\label{existence and uniqueness}
		There exists a unique triple of $\mathbb{G}$-adapted processes $(Y,Z_1,Z_2)$ which solves \eqref{bsde}. 
	\end{myth}
	\begin{proof}
		(Uniqueness.) Suppose $(Y^1,Z^1_1,Z^1_2)$ and $(Y^2,Z^2_1,Z^2_2)$ both are solutions to \eqref{bsde}. Then,
		\begin{align*}
			E\left|Y^1_t-Y^2_t\right|^2&=\int_t^T\bigg[-2E\left\{\left[Y^1_u-Y^2_u\right]\left[F(u,Y^1_u,Z^1_1(u),Z^1_2(u),X_u)\right.\right.\\
			&\quad\left.\left.-F(u,Y^1_u,Z^2_1(u),Z^2_2(u),X_u)\right]\right\}\\
			&\quad-E|Z^1_1(u)-Z^2_1(u)|^2-E\left\lVert Z^1_2(u)-Z^2_2(u)\right\rVert_{X_{u-}}^2\bigg]du.
		\end{align*}
		For any $k\in\mathbb{R}$,
		\begin{align*} 
			E\left|Y^1_t-Y^2_t\right|^2&\leq\int_t^T\bigg[-E|Z^1_1(u)-Z^2_1(u)|^2-E\left\lVert Z^1_2(u)-Z^2_2(u)\right\rVert_{X_{u-}}^2\\
			&\quad\qquad+k^{-2}E\left|Y^1_u-Y^2_u\right|^2+k^2E\left|F(u,Y^1_u,Z^1_1(u),Z^1_2(u),X_u)\right.\\
			&\quad\qquad\left.-F(u,Y^2_u,Z^2_1(u),Z^2_2(u),X_u)\right|^2\bigg]du\\
			&\leq\int_t^T\bigg[(k^{-2}+\mu k^2)E\left|Y^1_u-Y^2_u\right|^2\\
			&\quad\qquad+(\mu k^2-1)\bigg(E|Z^1_1(u)-Z^2_1(u)|^2\\
			&\quad\qquad+E\left\lVert Z^1_2(u)-Z^2_2(u)\right\rVert_{X_{u-}}^2\bigg)\bigg]du.
		\end{align*}
		Setting $k^2=(2\mu)^{-1}$,
		\begin{align}\label{k is 2mu^-1} 
			0 \leq E\left|Y^1_t-Y^2_t\right|^2&\leq\frac{1}{2}\int_t^T\bigg[(4\mu+1)E\left|Y^1_u-Y^2_u\right|^2-E|Z^1_1(u)-Z^2_1(u)|^2\nonumber\\
			&\qquad\qquad-E\left\lVert Z^1_2(u)-Z^2_2(u)\right\rVert_{X_{u-}}^2\bigg]du\\
			&\leq \frac{4\mu+1}{2}\int_t^TE\left|Y^1_u-Y^2_u\right|^2du.\nonumber
		\end{align}
		By Gr\"onwall's inequality, $E\left|Y^1_t-Y^2_t\right|^2=0$ which implies that $Y^1_t=Y^2_t$ $P$-a.s for all $t\in[0,T]$. In addition, from $\eqref{k is 2mu^-1}$,
		\begin{align*}
			&\int_t^T\bigg[E|Z^1_1(u)-Z^2_1(u)|^2+E\left\lVert Z^1_2(u)-Z^2_2(u)\right\rVert_{X_{u-}}^2\bigg]\\
			&\leq (4\mu+1)\int_t^TE\left|Y^1_u-Y^2_u\right|^2du=0.
		\end{align*}
		Thus, $Z^1_1(t)=Z^2_1(t)$ $P$-a.s and $Z^1_2(t)=Z^2_2(t)$ $d\langle M,M \rangle_t\times dP$-a.s for all $t\in[0,T]$. Hence, $(Y^1,Z^1_1,Z^1_2)$ and $(Y^2,Z^2_1,Z^2_2)$ are equivalent and \eqref{bsde} has a unique solution.
		
		(Existence.) Define an approximating sequence using a Picard-type iteration. Suppose that the triple $(Y^{n+1}_t,Z^{n+1}_1(t),Z^{n+1}_2(t))$ is the solution to the following BSDEBM:
		\begin{align}\label{bsde 1.1}
			Y^{n+1}_t&=Q+\int_t^TF(u,Y^{n}_{u},Z^{n+1}_1(u),Z^{n+1}_2(u),X_u)du-\int_t^TZ^{n+1}_1(u)dW_u\nonumber\\
			&\quad-\int_t^T(Z^{n+1}_2(u))^{\top}dM_u.
		\end{align}
		
		From Proposition \ref{BSDE with Z_1 and Z_2}, \eqref{bsde 1.1} has a unique solution. Furthermore,
		\begin{align*}
			E\left|Y^{n+1}_t-Y^{n}_t\right|^2&=-2\int_t^TE\bigg\{\left[Y^{n+1}_u-Y^{n}_u\right]\left[F(u,Y^{n}_u,Z^{n+1}_1(u),Z^{n+1}_2(u),X_u)\right.\\
			&\qquad\quad\left.-F(u,Y^{n-1}_u,Z^{n}_1(u),Z^{n}_2(u),X_u)\right]\bigg\}du\\
			&\qquad\quad-\int_t^TE|Z^{n+1}_1(u)-Z^{n}_1(u)|^2du\\
			&\qquad\quad-\int_t^TE\left\lVert Z^{n+1}_2(u)-Z^{n}_2(u)\right\rVert_{X_{u-}}^2du.
		\end{align*}
		Then for $k\in\mathbb{R}$,
		\begin{align*} 
			E\left|Y^{n+1}_t-Y^{n}_t\right|^2&\leq\int_t^T\bigg[-E|Z^{n+1}_1(u)-Z^{n}_1(u)|^2-E\left\lVert Z^{n+1}_2(u)-Z^{n}_2(u)\right\rVert_{X_{u-}}^2\\
			&\quad\qquad+k^{-2}E\left|Y^{n+1}_u-Y^{n}_u\right|^2\\
			&\quad\qquad+k^2E\left|F(u,Y^{n}_u,Z^{n+1}_1(u),Z^{n+1}_2(u),X_u)\right.\\
			&\quad\qquad\left.-F(u,Y^{n-1}_u,Z^{n}_1(u),Z^{n}_2(u),X_u)\right|^2\bigg]du\\
			&\leq\int_t^T\bigg[k^{-2}E\left|Y^{n+1}_u-Y^{n}_u\right|^2+\mu k^2E\left|Y^{n}_u-Y^{n-1}_u\right|^2\\
			&\qquad\quad+(\mu k^2-1)\left(E|Z^{n+1}_1(u)-Z^{n}_1(u)|^2\right.\\
			&\qquad\quad\left.+E\left\lVert Z^{n+1}_2(u)-Z^{n}_2(u)\right\rVert_{X_{u-}}^2\right)\bigg]du.
		\end{align*}
		Setting $k^2= (2\mu)^{-1}$,
		\begin{align}
			E\left|Y^{n+1}_t-Y^{n}_t\right|^2&\leq\frac{1}{2}\int_t^T\bigg[4\mu E\left|Y^{n+1}_u-Y^{n}_u\right|^2+E\left|Y^{n}_u-Y^{n-1}_u\right|^2\nonumber\\
			&\qquad\qquad-E|Z^{n+1}_1(u)-Z^{n}_1(u)|^2\nonumber\\
			&\qquad\qquad-E\left\lVert Z^{n+1}_2(u)-Z^{n}_2(u)\right\rVert_{X_{u-}}^2\bigg]du\label{inequality for Zeds}\\
			&\leq \frac{4\mu+1}{2}\int_t^T\bigg[E\left|Y^{n+1}_u-Y^{n}_u\right|^2-E\left|Y^{n}_u-Y^{n-1}_u\right|^2\bigg]du.\label{phi rep}
		\end{align}
		Define $\phi_n(t)=\int_t^TE|Y^{n}_u-Y^{n-1}_u|^2du$ and $\mu_0=(4\mu+1)/2$. Then \eqref{phi rep} can be written as
		\begin{align*}
			\frac{d\phi_{n+1}(t)}{dt}-\mu_0\phi_{n+1}(t)\leq\mu_0\phi_n(t)
		\end{align*}
		or, equivalently,
		\begin{align*}
			\phi_{n+1}(t)\leq\mu_0\int_t^Te^{\mu_0(s-t)}\phi_n(s)ds.
		\end{align*}
		Iteration of the above inequality yields,
		\begin{align*}
			\phi_{n+1}(t)\leq\frac{\left[T\mu_0e^{T\mu_0}\right]^n}{n!}\phi_1(t)
		\end{align*}
		and hence,
		\begin{align*}
			\int_t^TE|Y^{n+1}_u-Y^{n}_u|^2du\leq\frac{\left[T\mu_0e^{T\mu_0}\right]^n}{n!}\int_t^TE|Y^{1}_u-Y^{0}_u|^2du.
		\end{align*}
		This implies that $\{Y^n_t\}$ is a Cauchy sequence. From \eqref{inequality for Zeds},
		\begin{align*}
			&\int_t^T\bigg[E|Z^{n+1}_1(u)-Z^{n}_1(u)|^2+E\left\lVert Z^{n+1}_2(u)-Z^{n}_2(u)\right\rVert_{X_{u-}}^2\bigg]\\
			&\leq (4\mu+1)\int_t^T\bigg[E\left|Y^{n+1}_u-Y^{n}_u\right|^2-E\left|Y^{n}_u-Y^{n-1}_u\right|^2\bigg]du.
		\end{align*}
		This implies that $\{Z^n_1(t)\}$ and $\{Z^n_2(t)\}$ are also Cauchy sequences under appropriate norms. By completeness, the limits exist and thus satisfy \eqref{bsde}. 
	\end{proof}
	
	\section{Comparison Theorem}\label{comparison theorem chapter 5}
	
	This section develops a comparison theorem for the BSDEBM in \eqref{bsde}. We adapt some methods used in \cite{cohen:bsde}.
	
	\begin{mydef}
		A pair $(F,Q)$ is called standard if it satisfies the following conditions:
		\begin{enumerate}
			\item[(i)] $F$ is $P$-a.s left continuous in $t$.
			\item[(ii)] There exists a constant $\mu>0$ such that for all $(\omega,t,x)\in\Omega\times[0,T]\times\mc{S}$, and triples $(y,z_1,z_2),(y',z'_1,z'_2)\in\mathbb{R}\times\mathbb{R}\times\mathbb{R}^N$,
			\begin{align*}
				|F(t,y,z_1,z_2,x)-F(t,y',z'_1,z'_2,x)|^2&\leq \mu(|y-y'|^2+|z_1-z'_1|^2+\lVert z_2-z'_2\rVert^2_{X_t}).
			\end{align*}
			\item[(iii)] For  all triples $(Y,Z_1,Z_2)$ such that $E\int_0^T|Y_t|^2dt<\infty$, $E\int_0^T|Z_1(t)|^2dt<\infty$, and $E[\int_0^T\lVert Z_2(t)\rVert^2_{X_{t-}} dt]<\infty$,
			\begin{align*}
				E\left[\int_0^T|F(\omega,t,Y_{t-},Z_1(t),Z_2(t),X_t)|^2dt\right]<\infty.
			\end{align*}
			\item[(iv)] $Q\in L^2(\mathbb{R};\mc{G}_T)$.
		\end{enumerate}
	\end{mydef}
	
	\begin{mylm}\label{z norm is 0 iff z is 0}
		Let $Z_t\in\mathbb{R}^N$. Then $\lVert Z_t \rVert_{X_{t-}}=0$ if and only if $Z_t=0$.
	\end{mylm}
	\begin{proof}
		See Lemma 3.6 of \cite{cohen:bsde}.
	\end{proof}
	
	\begin{mylm}\label{epsilon inequality for z norm not 0}
		Let $Z_2\in\mathbb{R}^N$ be a solution to a standard BSDEBM of the form \eqref{bsde}. Suppose that for a given $t$, $\lVert Z_2(t) \rVert_{X_{t-}}\neq 0$ and that for some $0<\epsilon<c^{3/2}N^{-3/2}$, we have for all $j=1,\ldots,N$
		\begin{align*}
			Z_2(t)^{\top}\psi_te_j\geq -\epsilon\lVert Z_2(t) \rVert_{X_{t-}}.
		\end{align*}
		Then $Z_2(t)^{\top}\psi_tX_{t-}\leq -\epsilon\lVert Z_2(t) \rVert_{X_{t-}}$.
	\end{mylm}	
	\begin{proof}
		See Lemma 3.10 of \cite{cohen:bsde}.
	\end{proof}
	
	\begin{mypr}
		Let $(\rho,\alpha,\beta,\gamma)$ be a $dt\times P$-a.s. bounded $(\mathbb{R},\mathbb{R},\mathbb{R}^{1\times N},\mathbb{R})$-valued predictable process, $\phi$ a predictable $\mathbb{R}$-valued process with $E[\int_0^T|\phi_t|^2dt]<\infty$, $Q\in L^2(\mc{G}_T)$. Define $\psi^+_t$ as the Moore-Penrose inverse of $\psi_t$. Then the linear BSDEBM given by
		\begin{align*}
			Y_t&=Q+\int_t^T\left[\phi_u+\rho_uY_{u-}+\alpha_uZ_1(u)+\gamma_u\beta_uZ_2(u)\right]du\\
			&\quad-\int_t^TZ_1(u)dW_u-\int_t^TZ_2(u)^{\top}dM_u
		\end{align*}
		has a unique square integrable solution $(Y,Z_1,Z_2)$ that can be expressed as
		\begin{align*}
			Y_t=(\mc{E}(\Gamma)_t)^{-1}E\left[\mc{E}(\Gamma)_TQ+\int_t^T\mc{E}(\Gamma)_{u-}\phi_udu\bigg|\mc{G}_t\right],
		\end{align*} 
		where $\mc{E}(\cdot)$ is the Doleans-Dade exponential function and 
		\begin{align*}
			\Gamma_t=\int_0^t\rho_udu +\int_0^t\alpha_udW_u+\int_0^t \gamma_u\beta_u\psi_u^+dM_u.
		\end{align*}
	\end{mypr}
	\begin{proof}
		This follows from Theorem 19.2.2 of \cite{elliott:stochastic} and Theorem 3.11 of \cite{cohen:bsde}.
	\end{proof}
	
	\begin{mylm}\label{condition for beta}
		Suppose for $u\in(t,T]$, $\gamma_u\beta_u\psi_u^+(e_j-X_{u-})>-1$ for all $j$ such that $e_j^{\top}AX_{u-}>0$. Then $\Gamma_u\geq 0$ for all $u\in(t,T]$.
	\end{mylm}
	\begin{proof}
		See Theorem 3.16 of \cite{cohen:bsde}.
	\end{proof}
	
	\begin{mylm}\label{moore penrose property}
		For all $t\in[0,T]$ and $j$ such that $e_j^{\top}AX_{t-}\neq 0$,
		\begin{align*}
			\psi_t\psi_t^+(e_j-X_{t-})=\psi_t^+\psi_t(e_j-X_{t-})=(e_j-X_{t-})
		\end{align*}
	\end{mylm}
	\begin{proof}
		See Lemma 3.4 of \cite{cohen:bsde}.
	\end{proof}
	
	\begin{mylm}\label{psi ej}
		Let $\psi_t$ be the nonnegative definite matrix defined in \eqref{defn of psi}. Then for all $j$
		\begin{align*}
			\psi_te_j=\begin{cases}
				(e_j^{\top}AX_{t-})(e_j-X_{t-}) & \mbox{if $e_j\neq X_{t-}$}\\
				-AX_{t-} & \mbox{if $e_j=X_{t-}$}.
			\end{cases}
		\end{align*}
	\end{mylm}
	\begin{proof}
		From \eqref{defn of psi},
		\begin{align*}
			\psi_te_j&=\mbox{diag}(AX_{t-})e_j-A\mbox{diag}(X_{t-})e_j-\mbox{diag}(X_{t-})A^{\top}e_j\\
			&=e_j^{\top}AX_{t-}e_j-A\mbox{diag}(X_{t-})e_j-e_j^{\top}AX_{t-}X_{t-}
		\end{align*}
		If $e_j\neq X_{t-}$, the second term becomes 0 which proves the first case. If $e_j=X_{t-}$, the first and third term cancel each other out. Furthermore, $\mbox{diag}(X_{t-})X_{t-}=X_{t-}$, which proves the second case. 
	\end{proof}

	\begin{myth}[Comparison Theorem]\label{comparison theorem}
		Suppose that we have two BSDEBMs corresponding to two standard pairs $(F^1,Q^1)$ and $(F^2,Q^2)$. Let $(Y^1,Z^1_1,Z^1_2)$ and $(Y^2,Z^2_1,Z^2_2)$ be the associated solutions. Suppose further that the following conditions hold:
		\begin{enumerate}
			\item[(1)] $Q^1\geq Q^2$ $P$-a.s.
			\item[(2)] $F^1(t,Y_t^2,Z_1^2(t),Z_2^2(t),X_t)\geq F^2(t,Y_t^2,Z_1^2(t),Z_2^2(t),X_t)$ $dt\times P$-a.s.
			\item[(3)] For all $Z_2^1(t),Z_2^2(t)$, there exists an $\epsilon>0$ such that if
			\begin{align*}
				[Z_2^1(t)-Z_2^2(t)]^{\top}\psi_te_j\geq -\epsilon\lVert Z_2^1(t)-Z_2^2(t)\rVert_{X_{t-}}
			\end{align*}
			for all $j=1,\ldots,N$, then
			\begin{align*}
				&F^1(t,Y_t^2,Z_1^2(t),Z_2^1(t),X_t)-F^1(t,Y_t^2,Z_1^2(t),Z_2^2(t),X_t)\\ &\geq[Z_2^1(t)-Z_2^2(t)]^{\top}\psi_tX_{t-}
			\end{align*}
			with equality only if $\lVert Z_2^1(t)-Z_2^2(t)\rVert_{X_{t-}}=0$.
		\end{enumerate}
		Then $Y^1\geq Y^2$ $P$-a.s.
	\end{myth}
	\begin{proof}
		Given the two solutions $(Y^1,Z^1_1,Z^1_2)$ and $(Y^2,Z^2_1,Z^2_2)$, we can write
		\begin{align}\label{bsde difference}
			Y_t^1-Y_t^2&=Q^1-Q^2\nonumber\\
			&\quad+\int_t^T\left[F^1(u,Y^1_{u},Z^1_1,Z^1_2,X_u)-F^2(u,Y^2_{u},Z^2_1,Z^2_2,X_u)\right]du\nonumber\\
			&\quad-\int_t^T\left[Z^1_1(u)-Z^2_1(u)\right]dW_u-\int_t^T[Z^1_2(u)-Z^2_2(u)]^{\top}dM_u.
		\end{align}
		We assume that $\epsilon<c^{-3/2}N^{3/2}$. With the convention $0/0:=0$, we consider the following cases:
		
		(Case 1) Suppose $F^1(t,Y_t^2,Z_1^2(t),Z_2^1(t),X_t)-F^1(t,Y_t^2,Z_1^2(t),Z_2^2(t),X_t)\geq 0$. Then define
		\begin{align*}
			\phi_t&:=F^1(t,Y_t^2,Z_1^2(t),Z_2^1(t),X_t)-F^2(t,Y_t^2,Z_1^2(t),Z_2^2(t),X_t)\\
			\rho_t&:=\frac{F^1(t,Y_t^1,Z_1^2(t),Z_2^1(t),X_t)-F^1(t,Y_t^2,Z_1^2(t),Z_2^1(t),X_t)}{Y_t^1-Y_t^2}\\
			\alpha_t&:=\frac{F^1(t,Y_t^1,Z_1^1(t),Z_2^1(t),X_t)-F^1(t,Y_t^1,Z_1^2(t),Z_2^1(t),X_t)}{Z_1^1(X_t)-Z_1^2(X_t)}\\
			\beta_t&:=0\\
			\gamma_t&:=1.
		\end{align*}
		Here, $0\in\mathbb{R}^{1\times N}$. Since $\phi_t$ can be written as
		\begin{align*}
			\phi_t&=F^1(t,Y_t^2,Z_1^2(t),Z_2^2(t),X_t)-F^2(t,Y_t^2,Z_1^2(t),Z_2^2(t),X_t)\\
			&\quad+F^1(t,Y_t^2,Z_1^2(t),Z_2^1(t),X_t)-F^1(t,Y_t^2,Z_1^2(t),Z_2^2(t),X_t),
		\end{align*}
		then by assumption (2), $\phi_t\geq 0$.
		
		(Case 2) Suppose $F^1(t,Y_t^2,Z_1^2(t),Z_2^1(t),X_t)-F^1(t,Y_t^2,Z_1^2(t),Z_2^2(t),X_t)<0$ and there exists some $j=1,\ldots,N$ such that
		\begin{align*}
			[Z_2^1(t)-Z_2^2(t)]^{\top}\psi_te_j<-\epsilon\lVert Z_2^1(t)-Z_2^2(t)\rVert_{X_{t-}}.
		\end{align*}
		Then define
		\begin{align*}
			\phi_t&:=F^1(t,Y_t^2,Z_1^2(t),Z_2^2(t),X_t)-F^2(t,Y_t^2,Z_1^2(t),Z_2^2(t),X_t)\\
			\rho_t&:=\frac{F^1(t,Y_t^1,Z_1^2(t),Z_2^1(t),X_t)-F^1(t,Y_t^2,Z_1^2(t),Z_2^1(t),X_t)}{Y_t^1-Y_t^2}\\
			\alpha_t&:=\frac{F^1(t,Y_t^1,Z_1^1(t),Z_2^1(t),X_t)-F^1(t,Y_t^1,Z_1^2(t),Z_2^1(t),X_t)}{Z_1^1(t)-Z_1^2(t)}\\
			\beta_t&:=\frac{F^1(t,Y_t^2,Z_1^2(t),Z_2^1(t),X_t)-F^1(t,Y_t^2,Z_1^2(t),Z_2^2(t),X_t)}{[Z_2^1(t)-Z_2^2(t)]^{\top}\psi_te_j}e_j^{\top}\psi_t\\
			\gamma_t&:=1.
		\end{align*}
		It follows directly from assumption (2) that $\phi_t\geq 0$.
		
		(Case 3) Suppose $F^1(t,Y_t^2,Z_1^2(t),Z_2^1(t),X_t)-F^1(t,Y_t^2,Z_1^2(t),Z_2^2(t),X_t)<0$ and for all $j=1,\ldots,N$,
		\begin{align*}
			[Z_2^1(t)-Z_2^2(t)]^{\top}\psi_te_j\geq-\epsilon\lVert Z_2^1(t)-Z_2^2(t)\rVert_{X_{t-}}.
		\end{align*}
		Then define
		\begin{align*}
			\phi_t&:=F^1(t,Y_t^2,Z_1^2(t),Z_2^2(t),X_t)-F^2(t,Y_t^2,Z_1^2(t),Z_2^2(t),X_t)\\
			\rho_t&:=\frac{F^1(t,Y_t^1,Z_1^2(t),Z_2^1(t),X_t)-F^1(t,Y_t^2,Z_1^2(t),Z_2^1(t),X_t)}{Y_t^1-Y_t^2}\\
			\alpha_t&:=\frac{F^1(t,Y_t^1,Z_1^1(t),Z_2^1(t),X_t)-F^1(t,Y_t^1,Z_1^2(t),Z_2^1(t),X_t)}{Z_1^1(t)-Z_1^2(t)}\\
			\beta_t&:=\frac{F^1(t,Y_t^2,Z_1^2(t),Z_2^1(t),X_t)-F^1(t,Y_t^2,Z_1^2(t),Z_2^2(t),X_t)}{[Z_2^1(t)-Z_2^2(t)]^{\top}\psi_tX_{t-}}X_{t-}^{\top}\psi_t\\
			\gamma_t&:=1.
		\end{align*}
		It also follows directly from assumption (2) that $\phi_t\geq 0$.
		
		Write $\delta Y=Y^1-Y^2$, $\delta Z_1=Z_1^1-Z_1^2$, and $\delta Z_2=Z_2^1-Z_2^2$. Then for all three cases, we can write \eqref{bsde difference} as
		\begin{align*}
			\delta Y_t&=Q^1-Q^2+\int_t^T\left[\phi_u+\rho_u\delta Y_u +\alpha_u \delta Z_1(u)+\gamma_u\beta_u\delta Z_2(u)\right]du\nonumber\\
			&\quad+\int_t^T\delta Z_1(u)dW_u+\int_t^T\delta Z_2(u)^{\top}dM_u.
		\end{align*}
		
		Since $F^i$ is standard for $i=1,2$, then $E[\int_0^T|\phi_t|^2dt]<\infty$ for all three cases. In addition, $\rho_t$ and $\alpha_t$ are $dt\times P$-a.s. bounded by Lipschitz continuity, and trivially $\gamma_t$ is bounded for all three cases.
		
		In Case 1, $\beta_t=0$ and hence $\beta_t$ is bounded. In Case 2, by assumption,
		\begin{align*}
			\left|[Z_2^1(t)-Z_2^2(t)]^{\top}\psi_te_j\right|>\epsilon\lVert Z_2^1(t)-Z_2^2(t)\rVert_{X_{t-}}.
		\end{align*}
		Hence,
		\begin{align*}
			|\beta_t|<\frac{|F^1(t,Y_t^2,Z_1^2(t),Z_2^1(t),X_t)-F^1(t,Y_t^2,Z_1^2(t),Z_2^2(t),X_t)|}{\epsilon\lVert Z_2^1(t)-Z_2^2(t)\rVert_{X_{t-}}}|e_j^{\top}\psi_t|,
		\end{align*}
		which implies that $\beta_t$ is bounded by Lipschitz continuity.
		
		In Case 3, since $F^1(t,Y_t^2,Z_1^2(t),Z_2^1(t),X_t)\neq F^1(t,Y_t^2,Z_1^2(t),Z_2^2(t),X_t)$ by assumption, then $Z_2^1(t)\neq Z_2^2(t)$. By Lemma \ref{z norm is 0 iff z is 0}, $\lVert Z_2^1(t)-Z_2^2(t) \rVert_{X_{t-}}\neq 0$. In addition, by Lemma \ref{epsilon inequality for z norm not 0},
		\begin{align*}
			[Z_2^1(t)-Z_2^2(t)]^{\top}\psi_tX_{t-}]<-\epsilon\lVert Z_2^1(t)-Z_2^2(t)\rVert_{X_{t-}}.
		\end{align*}
		Similarly, $\beta_t$ in this case is also bounded by Lipschitz continuity.
		
		By Lemma \ref{condition for beta}, it suffices to show that $\beta_u\psi_u^+(e_k-X_{u-})>-1$ for all $k$ such that $e_k^{\top}AX_{u-}>0$. In Case 1, $\gamma_u\beta_u\psi_u^+(e_k-X_{u-})=0>-1$.
		
		In Case 2, by Lemma \ref{moore penrose property},
		\begin{align*}
			\beta_t\psi_t^+(e_k-X_{t-})&=\frac{F^1(t,Y_t^2,Z_1^2(t),Z_2^1(t),X_t)-F^1(t,Y_t^2,Z_1^2(t),Z_2^2(t),X_t)}{[Z_2^1(t)-Z_2^2(t)]^{\top}\psi_te_j}\\
			&\quad\times e_j^{\top}\psi_t\psi_t^+(e_k-X_{t-})\\
			&=\frac{F^1(t,Y_t^2,Z_1^2(t),Z_2^1(t),X_t)-F^1(t,Y_t^2,Z_1^2(t),Z_2^2(t),X_t)}{[Z_2^1(t)-Z_2^2(t)]^{\top}\psi_te_j}\\
			&\quad\times e_j^{\top}(e_k-X_{t-})\\
		\end{align*}
		If $k\neq j$, then $e_j^{\top}(e_k-X_{t-})=0$. Otherwise, $e_j^{\top}(e_k-X_{t-})=1$. By assumption,
		\begin{align*}
			\frac{F^1(t,Y_t^2,Z_1^2(t),Z_2^1(t),X_t)-F^1(t,Y_t^2,Z_1^2(t),Z_2^2(t),X_t)}{[Z_2^1(t)-Z_2^2(t)]^{\top}\psi_te_j}>0.
		\end{align*}
		Hence, $\beta_t\psi_t^+(e_k-X_{t-})\geq 0$.
		
		In Case 3, by Lemma \ref{moore penrose property}, Lemma \ref{psi ej}, and assumption (3),
		\begin{align*}
			\beta_t\psi_t^+(e_k-X_{t-})&=\frac{F^1(t,Y_t^2,Z_1^2(t),Z_2^1(t),X_t)-F^1(t,Y_t^2,Z_1^2(t),Z_2^2(t),X_t)}{[Z_2^1(t)-Z_2^2(t)]^{\top}\psi_tX_{t-}}\\
			&\quad\times X_{t-}^{\top}\psi_t\psi_t^+(e_k-X_{t-})\\
			&=\frac{F^1(t,Y_t^2,Z_1^2(t),Z_2^1(t),X_t)-F^1(t,Y_t^2,Z_1^2(t),Z_2^2(t),X_t)}{[Z_2^1(t)-Z_2^2(t)]^{\top}\psi_te_j}\\
			&\quad\times X_{t-}^{\top}(e_k-X_{t-})\\
			&=-\frac{F^1(t,Y_t^2,Z_1^2(t),Z_2^1(t),X_t)-F^1(t,Y_t^2,Z_1^2(t),Z_2^2(t),X_t)}{[Z_2^1(t)-Z_2^2(t)]^{\top}\psi_tX_{t-}}\\
			&>-\frac{[Z_2^1(t)-Z_2^2(t)]^{\top}\psi_tX_{t-}}{[Z_2^1(t)-Z_2^2(t)]^{\top}\psi_tX_{t-}}=-1.
		\end{align*}
		
		Since $Q^1-Q^2\geq 0$, it then follows from Lemma \ref{condition for beta} that $Y_t^1-Y_t^2=\delta Y_t\geq 0$.
	\end{proof}
	
	\section{Sublinear Expectations and Bid-Ask Pricing}\label{sublinear evaluations and expectations}
	
	In this section, we define filtration consistent sublinear evaluations and expectations. Sublinear expectations are then shown to be solutions of the BSDEBM in \eqref{bsde}. Representations for the bid and ask prices are also presented.
	
	\begin{mydef}[Sublinear Evaluations]
		A system of operators
		\begin{align*}
			\mc{E}_{s,t}:L^2(\mc{F}_t)\to L^2(\mc{F}_s),\quad 0\leq s\leq t\leq T
		\end{align*}
		is called an $\mc{F}_t$-consistent sublinear evaluation for $\{\mc{Q}_{s,t}\subset L^2(\mc{F}_t)|0\leq s\leq t\leq T\}$ if it satisfies the following properties:
		\begin{enumerate}
			\item For $Q^1,Q^2\in\mc{Q}_{s,t}$, if $Q^1\geq Q^2$ $P$-a.s., then $\mc{E}_{s,t}(Q^1)\geq\mc{E}_{s,t}(Q^2)$ $P$-a.s., with equality iff $Q^1=Q^2$ $P$-a.s.
			\item $\mc{E}_{t,t}(Q)=Q$ $P$-a.s.
			\item For $r\leq s\leq t$, $\mc{E}_{r,s}(\mc{E}_{s,t}(Q))=\mc{E}_{r,t}(Q)$ $P$-a.s.
			\item For any $A\in\mc{F}_s$, $1_A\mc{E}_{s,t}(Q)=1_A\mc{E}_{s,t}(1_AQ)$ $P$-a.s.
			\item For $Q^1,Q^2\in\mc{Q}_{s,t}$, $\mc{E}_{s,t}(Q^1+Q^2)\leq\mc{E}_{s,t}(Q^1)+\mc{E}_{s,t}(Q^2)$ $P$-a.s.
			\item For $0\leq\lambda\in L^2(\mc{F}_s)$ and $\lambda Q\in\mc{Q}_{s,t}$, $\mc{E}_{s,t}(\lambda Q)=\lambda\mc{E}_{s,t}(Q)$ $P$-a.s.
		\end{enumerate} 
	\end{mydef}
	
	\begin{mydef}[Balanced Drivers]
		Suppose the family of sets $\{\mc{Q}_{s,t}\subset L^2(\mc{F}_t)\}$ are nondecreasing in $s$. For a fixed driver $F$, suppose further that, for all $s\leq t\leq T$, condition $(3)$ of Theorem \ref{comparison theorem} holds on $(s,t]$ for all $Q^1,Q^2\in\mc{Q}_{s,t}$. Then $F$ is called a balanced driver on $\{\mc{Q}_{s,t}\}$. 
	\end{mydef}
	
	\begin{myth}\label{sublinear evaluation theorem}
		Fix a balanced driver $F$. Consider a collection of sets $\{\mc{Q}_{s,t}\subset L^2(\mc{F}_t)\}$ with $\mc{Q}_{r,s}\subseteq\mc{Q}_{s,t}$ for all $r\leq s\leq t$. For $s\leq t$, define the conditional $F$-evaluation by
		\begin{align*}
			\mc{E}^F_{s,t}(Q)=Y_s,
		\end{align*}
		where $Y_s$ is the solution to
		\begin{align*}		
			Y_s-\int_s^tF(u,Y_u,Z_1(u),Z_2(u),X_u)du+\int_s^tZ_1(u)dW_u+\int_s^tZ^{\top}_2(u)dM_u=Q.
		\end{align*}
		Suppose $F$ satisfies the following assumptions for each $(t,x)\in[0,T]\times\mc{S}$:
		\begin{enumerate}
			\item[i.] (Subadditivity)
			\begin{align*}
				&F(t,Y^1_{t-}+Y^2_{t-},Z_1^1(t)+Z_1^2(t),Z_2^1(t)+Z_2^2(t),x)\\
				&\leq F(t,Y^1_{t-},Z_1^1(t),Z_2^1(t))+F(t,Y^2_{t-},Z_1^2(t),Z_2^2(t),x)
			\end{align*}
			\item[ii.] (Positive Homogeneity) For $\lambda\geq 0$,
			\begin{align*}
				F(t,\lambda Y_{t-},\lambda Z_1(t),\lambda Z_2(t),x)&= \lambda F(t,Y_{t-},Z_1(t),Z_2(t),x).
			\end{align*}
		\end{enumerate}
		Then $\mc{E}^F_{s,t}$ is an $\mc{F}_t$-consistent sublinear evaluation for $\{\mc{Q}_{s,t}\}$.
	\end{myth}
	\begin{proof}
		It suffices to show that $\mc{E}^F_{s,t}$ satisfies the six properties of an $\mc{F}_t$-consistent sublinear evaluation.
		\begin{enumerate}
			\item Suppose $Q^1\geq Q^2$. Since the driver $F$ is balanced, then by Theorem \ref{comparison theorem},
			\begin{align*}
				\mc{E}^F_{s,t}(Q^1)=Y^1_s\geq Y^2_s=\mc{E}^F_{s,t}(Q^2).
			\end{align*}
			
			\item By definition,
			\begin{align*}
				\mc{E}^F_{t,t}(Q)=Y_t&=Q.
			\end{align*}
			
			\item For any $r\leq s\leq t$, let $Y$ be the solution to the relevant BSDEBM. Then,
			\begin{align*}
				Y_s&=Y_r-\int_r^sF(u,Y_u,Z_1(X),Z_2(u),X_u)du+\int_r^sZ_1(u)dW_u\\
				&\quad+\int_r^sZ^{\top}_2(u)dM_u.
			\end{align*}
			Hence, $\mc{E}^F_{r,s}(Y_s)=Y_r$. By definition, we then have
			\begin{align*}
				\mc{E}^F_{r,s}(\mc{E}^F_{s,t}(Q))=\mc{E}^F_{r,t}(Q).
			\end{align*}
			
			\item Let $A\in\mc{F}_s$. By definition, $\mc{E}^F_{s,t}(Q)$ and $\mc{E}^F_{s,t}(1_AQ)$ are the respective solutions to
			\begin{align*}
				Q&=Y_s-\int_s^tF(u,Y_u,Z_1(u),Z_2(u),X_u)du+\int_s^tZ_1(u)dW_u\\
				&\quad+\int_s^tZ^{\top}_2(u)dM_u
			\end{align*}
			and
			\begin{align*}
				1_AQ&=Y_s-\int_s^tF(u,Y_u,Z_1(u),Z_2(u),X_u)du+\int_s^tZ_1(u)dW_u\\
				&\quad+\int_s^tZ^{\top}_2(u)dM_u.
			\end{align*}
			Premultiplying the first equation by $1_A$ yields a BSDEBM with terminal value $1_AQ$, solution $1_A\mc{E}^F_{s,t}(Q)$, and driver $1_AF$. Doing the same thing to the second equation yields a BSDEBM with terminal value $1_AQ$, solution $1_A\mc{E}^F_{s,t}(1_AQ)$, and driver $1_AF$. Now by Theorem \ref{existence and uniqueness}, the solution of the BSDEBM is unique. That is,
			\begin{align*}
				1_A\mc{E}^F_{s,t}(Q)=1_A\mc{E}^F_{s,t}(1_AQ).
			\end{align*}
			
			\item Consider two BSDEBMs with terminal conditions $Q^1$ and $Q^2$. Then,
			\begin{align*}
				Q^1+Q^2&=Y^1_s+Y^2_s\\
				&\quad-\int_s^t\left[F(u,Y^1_{u},Z^1_1,Z^1_2,X_u)+F(u,Y^2_{u},Z^2_1,Z^2_2,X_u)\right]du\\
				&\quad+\int_s^t\left[Z^1_1(u)+Z^2_1(u)\right]dW_u+\int_s^t\left[Z^1_2(u)+Z^2_2(u)\right]^{\top}dM_u
			\end{align*}
			We can then say that the above BSDEBM has terminal condition $Q^1+Q^2$, solution $\mc{E}^F_{s,t}(Q^1)+\mc{E}^F_{s,t}(Q^2)$, and driver $F(t,Y^1_{t-},Z^1_1(t),Z^1_2(t),X_t)+F(t,Y^2_{t-},Z^2_1(t),Z^2_2(t),X_t)$. Consider another BSDEBM with terminal condition $Q^1+Q^2$ and driver $F(t,Y^1_{t-}+Y^2_{t-},Z_1^1(t)+Z_1^2(t),Z_2^1(t)+Z_2^2(t),X_t)$. Denote the solution to this BSDEBM by $\mc{E}^F_{s,t}(Q^1+Q^2)$. The three conditions under Theorem \ref{comparison theorem} are satisfied. Condition (1) is trivially satisfied, condition (2) is satisfied by the subadditivity of $F$, and condition (3) is satisfied by the balanced driver assumption. Thus by Theorem \ref{comparison theorem},
			\begin{align*}
				\mc{E}^F_{s,t}(Q^1+Q^2)\leq\mc{E}^F_{s,t}(Q^1)+\mc{E}^F_{s,t}(Q^2).
			\end{align*}
			
			\item By definition, $\mc{E}^F_{s,t}(Q)$ is the solution to
			\begin{align*}
				Q&=Y_s-\int_s^tF(u,Y_u,Z_1(u),Z_2(u),X_u)du+\int_s^tZ_1(u)dW_u\\
				&\quad+\int_s^tZ^{\top}_2(u)dM_u.
			\end{align*}
			Multiplying by $\lambda\geq 0$, with $\lambda\in L^2(\mc{F}_s)$, and by the positive homogeneity assumption,
			\begin{align*}
				\lambda Q&=\lambda Y_s-\int_s^t\lambda F(u,Y_u,Z_1(u),Z_2(u),X_u)du+\int_s^t\lambda Z_1(u)dW_u\\
				&\quad+\int_s^t\lambda Z^{\top}_2(u)dM_u\\
				&=\lambda Y_s-\int_s^tF(u,\lambda Y_u,\lambda Z_1(u),\lambda Z_2(u),X_u)du+\int_s^t\lambda Z_1(u)dW_u\\
				&\quad+\int_s^t\lambda Z^{\top}_2(u)dM_u.
			\end{align*}
			We can then say that this BSDEBM has terminal condition $\lambda Q$ and solution $\lambda\mc{E}^F_{s,t}(Q)$. By Theorem \ref{existence and uniqueness},
			\begin{align*}
				\mc{E}^F_{s,t}(Q)=\lambda\mc{E}^F_{s,t}(Q).
			\end{align*}
		\end{enumerate}
	\end{proof}	
	
	\begin{mydef}[Sublinear Expectations]
		A system of operators
		\begin{align*}
			\mc{E}(\cdot|\mc{F}_t):L^2(\mc{F}_T)\to L^2(\mc{F}_t),\quad 0\leq t\leq T
		\end{align*}
		is called an $\mc{F}_t$-consistent sublinear expectation for $\{\mc{Q}_{t}\subset L^2(\mc{F}_T)|0\leq t\leq T\}$ if it satisfies the following properties:
		\begin{enumerate}
			\item For $Q^1,Q^2\in\mc{Q}_{t}$, if $Q^1\geq Q^2$ $P$-a.s., then $\mc{E}(Q^1|\mc{F}_t)\geq\mc{E}(Q^2|\mc{F}_t)$ $P$-a.s., with equality iff $Q^1=Q^2$ $P$-a.s.
			\item For any $\mc{F}_t$-measurable $Q$, $\mc{E}(Q|\mc{F}_t)=Q$ $P$-a.s.
			\item For $s\leq t$, $\mc{E}(\mc{E}(Q|\mc{F}_t)|\mc{F}_s)=\mc{E}(Q|\mc{F}_s)$ $P$-a.s.
			\item For any $A\in\mc{F}_t$, $1_A\mc{E}(Q|\mc{F}_t)=\mc{E}(1_AQ|\mc{F}_t)$ $P$-a.s.
			\item For $Q^1,Q^2\in\mc{Q}_{t}$, $\mc{E}(Q^1+Q^2|\mc{F}_t)\leq\mc{E}(Q^1|\mc{F}_t)+\mc{E}(Q^2|\mc{F}_t)$ $P$-a.s.
			\item For $0\leq\lambda\in L^2(\mc{F}_t)$ and $\lambda Q\in\mc{Q}_{t}$, $\mc{E}(\lambda Q|\mc{F}_t)=\lambda\mc{E}(Q|\mc{F}_t)$ $P$-a.s.
		\end{enumerate} 
	\end{mydef}	
	
	\begin{mylm}
		Any sublinear expectation is also a sublinear evaluation.
	\end{mylm} 
	\begin{proof}
		Setting $\mc{E}_{s,t}(\cdot)=\mc{E}(\cdot|\mc{F}_s)$ for $s\leq t$ proves the result.
	\end{proof}
	
	\begin{myth}
		Fix a balanced driver $F$ with $F(t,Y_{t-},0,0,x)=0$ for all $x\in\mc{S}$ $dt\times P$-a.s. Consider a collection of sets $\{\mc{Q}_{t}\subset L^2(\mc{F}_T)\}$ with $\mc{Q}_{s}\subseteq\mc{Q}_{t}$ for all $s\leq t$. Define the functional $\mc{E}^F(\cdot|\mc{F}_t)$ for each $t$ by
		\begin{align*}
			\mc{E}^F(Q|\mc{F}_t)=Y_t,
		\end{align*}
		where $Y_t$ is the solution to \eqref{bsde}. Suppose $F$ satisfies subaddivity and positive homogeneity. Then $\mc{E}^F(\cdot|\mc{F}_t)$ is an $\mc{F}_t$-consistent sublinear expectation for $\{\mc{Q}_{t}\}$.
	\end{myth}
	\begin{proof}
		It suffices to show that $\mc{E}^F(\cdot|\mc{F}_t)$ satisfies the six properties of an $\mc{F}_t$-consistent sublinear expectation. Properties 1, 3, 5, and 6 follow exactly as in Theorem \ref{sublinear evaluation theorem}. 
		\begin{enumerate}
			\item[2.] If $Q$ is $\mc{F}_t$-measurable, then taking an $\mc{F}_t$ conditional expectation implies that the solution of \eqref{bsde} has $Z_1(u)=0$ $dt\times P$-a.s. and $Z_2(u)=0$ $d\langle M,M\rangle_u\times P$-a.s. Hence, for all $t\leq u\leq T$,
			\begin{align*}
				F(u,Y_{u-},Z_1(u),Z_2(u),X_u)=0
			\end{align*} 
			by assumption and therefore $Y_t=Q$ or $\mc{E}^F(Q|\mc{F}_t)=Q$.
			
			\item[4.] Let $A\in\mc{F}_t$. By definition, $1_A\mc{E}^F(Q|\mc{F}_t)$ and $\mc{E}^F(1_AQ|\mc{F}_t)$ are the respective solutions to
			\begin{align*}
				1_AQ&=Y_t-\int_t^T1_AF(u,Y_u,Z_1(u),Z_2(u),X_u)du+\int_t^T1_AZ_1(u)dW_u\\
				&\quad+\int_t^T1_AZ^{\top}_2(u)dM_u
			\end{align*}
			and
			\begin{align*}
				1_AQ&=Y_t-\int_t^TF(u,Y_u,Z_1(u),Z_2(u),X_u)du+\int_t^TZ_1(u)dW_u\\
				&\quad+\int_t^TZ^{\top}_2(u)dM_u.
			\end{align*}
			Taking an $\mc{F}_t$ conditional expectation to the second equation yields
			\begin{align*}
				Z_1(u)=1_AZ_1(u)\quad\mbox{and}\quad Z_2(u)=1_AZ_2(u)
			\end{align*}
			for all $t\leq u\leq T$. In addition, by assumption, $F(t,Y_{t-},0,0,x)=0$ for all $x\in\mc{S}$ and hence,
			\begin{align*}
				1_AF(u,Y_u,Z_1(u),Z_2(u),X_u)=F(u,Y_u,1_AZ_1(u),1_AZ_2(u),X_u).
			\end{align*}
			Thus, both equations solve the same BSDEBM, and by Theorem \ref{existence and uniqueness}, 
			\begin{align*}
				1_A\mc{E}^F(Q|\mc{F}_t)=\mc{E}^F(1_AQ|\mc{F}_t).
			\end{align*}
		\end{enumerate}
	\end{proof}
	
	\begin{myth}\label{sublinear as supremum}
		Suppose $\mc{Q}_t$ is a linear space of real-valued functions defined on $\Omega$. Then $\mc{E}(\cdot|\mc{F}_t)$ is a sublinear expectation on $\{\mc{Q}_t\}$ if and only if there exists a family of linear expectations $E_{\theta}(\cdot|\mc{F}_t):\mc{Q}_t\mapsto\mathbb{R}$. indexed by $\theta\in\Theta$, such that for $Q\in\mc{Q}_t$,
		\begin{align*}
			\mc{E}(Q|\mc{F}_t)=\sup_{\theta\in\Theta}E_{\theta}(Q|\mc{F}_t).
		\end{align*}
	\end{myth}
	\begin{proof}
		The proof follows directly from Theorem 1.2.1 of \cite{peng:sublinear} and Theorem 1.1 of \cite{ref2:sublinear}.
	\end{proof}
	
	Conic finance, pioneered in \cite{ref1:conic}, deals with determining the bid and ask prices, which are commonly known as the buying and selling prices, respectively. The formulation is as follows. As in \cite{ref2:conic}, any cone of acceptable claims, denoted by $\mc{A}$, can be defined by a convex set of probability measures $\mc{M}$ equivalent to $P$. That is,
	\begin{equation*}
		\mc{A}:=\{Z\in L^1(\Omega,\mathcal{F},P)|E^{Q}[Z]\geq 0,  \forall Q\in\mc{M} \},
	\end{equation*}	
	where $L^1(\Omega,\mathcal{F},P)$ is the set of integrable random variables. The bid price and the ask price of the claim $Z$ at time 0 are then respectively given by
	\begin{equation*}
		b(Z):=\inf_{Q\in \mc{M}}E^{Q}[Z] 
	\end{equation*}
	and
	\begin{equation*}
		a(Z):=\sup_{Q\in \mc{M}}E^{Q}[Z].
	\end{equation*}
	The bid and ask prices are formulated above as the infimum and supremum, respectively, of the expected value of a risk or claim over all probability measures equivalent to the real-world probability measure. Given this representation, we have the following characterization of bid and ask prices.
	
	\begin{mypr}
		The bid and ask prices can be represented as the solutions to the BSDEBMs with driver $F$ and respective terminal conditions $-Q$ and $Q$, and solutions $-\mc{E}(-Q|\mc{F}_t)$ and $\mc{E}(Q|\mc{F}_t)$.
	\end{mypr}
	\begin{proof}
		The result follows from Theorem \ref{sublinear as supremum}.
	\end{proof}
	
	\section{Conclusion}
	The regime-switching BSDE \eqref{bsde} driven by Brownian motion and a Markov chain is shown to have a unique solution given by a triple $(Y,Z_1,Z_2)$ of $\mathbb{G}$-adapted processes. A comparison theorem also holds, given the assumptions for the driver $F$ and the terminal condition $Q$. The comparison result is then used to prove that the sublinear expectation of the form $\mc{E}^F(Q|\mc{F}_t)=Y_t$ is the solution to the BSDE \eqref{bsde} provided that the driver further satisfies subadditivity and positive homogeneity conditions. Following from the result that a sublinear expectation is the supremum of a family of linear expectations, the bid and ask prices are then represented as $-\mc{E}(-Q|\mc{F}_t)$ and $\mc{E}(Q|\mc{F}_t)$, respectively.
	
	Since control problems can be solved by considering BSDEs and some risk measures can be expressed as nonlinear expectations, future research could involve providing a connection among the these concepts given a certain set of dynamics for the regime-switching price process. Other dynamics such as regime-switching jump-diffusion systems or L\'{e}vy processes could also be explored.

	\section*{Acknowledgements}
	The authors would like to thank the Australian Research Council and the NSERC for continuing support.
	
	\bibliographystyle{siam}
	\bibliography{UpdatedReferences}

\end{document}